\documentclass{article}
\usepackage{amsmath}
\usepackage{amsfonts}
\usepackage{amsthm}
\usepackage{amssymb}
\usepackage[english]{babel}
\usepackage[ansinew]{inputenc}
\usepackage[all]{xy}
\usepackage{color}

\newtheorem{defi}{Definition}[section]
\newtheorem{theorem}[defi]{Theorem}
\newtheorem{corollary}[defi]{Corollary}
\newtheorem{lemma}[defi]{Lemma}
\newtheorem{proposition}[defi]{Proposition}

\newtheorem{definition}[defi]{Definition}

\input xy
\xyoption{all}

\title{The Fundamental Group Scheme of a non Reduced Scheme}

\author{Marco Antei}

\begin{document}
\maketitle

\textbf{Abstract.} We extend the definition of fundamental group scheme to non reduced schemes over any connected Dedekind scheme. Then we compare the fundamental group scheme of an affine scheme with that of its reduced part.

\tableofcontents
\bigskip

\section{Introduction}
The fundamental group scheme of a connected and reduced scheme $X$ over a field $k$ has been introduced by Madhav Nori in \cite{Nor} and \cite{Nor2}.  Later in \cite{GAS} Gasbarri  has shown how to construct the fundamental group scheme of an integral scheme $X$ over a connected Dedekind scheme $S$. Because of these unpleasant assumptions, as first pointed out by Nori in \cite{Nor2} \S 1,  we were not even able, in general, to construct the fundamental group scheme of a scheme $Y$ where  $Y\to X$ is a $G$-torsor for a finite and flat $S$-group scheme $G$ since $Y$ can easily be non reduced when $G$ is not \'etale. In $\cite{Gar}$, Theorem 3, Garuti has shown how to solve this problem when $S$ is both the spectrum of a field and a connected Dedekind scheme. A different proof when the base scheme is the spectrum of a field can be found in \cite{AnEm}, Theorem 4.1. The aim of this paper is to prove the following more general statement:

\begin{theorem}\label{teoPrincipe}(cf. Theorem \ref{teoGenerFundGrupK} and Theorem \ref{teoGenerFundGrupS}) Let $S$ be a connected Dedekind scheme, $X$ a connected scheme and $f:X\to S$ a faithfully flat morphism of finite type provided with a section $x\in X(S)$. Then $X$ has a fundamental group scheme.
\end{theorem} 

Although a tannakian description would be very interesting, the proof of theorem \ref{teoPrincipe} will not make use of tannakian formalism but will be based on a very simple idea: fix a connected Dedekind scheme $S$, in both Nori and Gasbarri's definitions the fundamental group scheme turns out to be the projective limit of all the finite and flat $S$-group schemes that act on torsors over $X$.  In both cases they  proved that the category of all finite  torsors (i.e. under the action of a finite and flat group scheme) over $X$ pointed over a section $x\in X(S)$ is filtered; then one can ask whether every torsor can be preceded by a ``privileged'' torsor (these are called \textit{reduced} by Nori when $S$ is the spectrum of a field). Here the idea is to inverse these operations: first we prove that every torsor is preceded by a somewhat ``privileged'' torsor, then we will prove that the category of such torsors is filtered, thus obtaining the desired result. The role of ``privileged'' torsors will be played by ``dominated'' torsors (cf. definition \ref{defDomi}).\\

Let moreover $X$ be a noetherian $S$-scheme, $X_{\text{red}}$ its reduced part and $x$ a geometric point of $X_{\text{red}}$ then one obtains, as a consequence of \cite{SGA1}, I, Th\'eor\`eme 8.3, an isomorphism $\pi_1^{\text{\'et}}(X_{\text{red}},x)\simeq \pi_1^{\text{\'et}}(X,x)$ between \'etale fundamental groups, induced by the canonical closed immersion $X_{\text{red}}\hookrightarrow X$.
In section \ref{sez:Comparison} we will study the behavior of the fundamental group scheme after a thickening of order one, using the theory of the cotangent complex. This will certainly provide useful tools for further analysis. At the moment we are able to compare the fundamental group scheme of an affine and noetherian scheme $X$ with that of its reduced part $X_{\text{red}}$ obtaining the following result:

\begin{theorem}\label{teoPrincipe2}(cf. Theorem \ref{theoRed}) Let $X$ be a connected, affine and noetherian scheme, faithfully flat and of finite type over $S$ and let $i_{\text{red}}:X_{\text{red}}\hookrightarrow X$ be the canonical closed immersion. Let $x\in X_{\text{red}}(S)$, then the natural morphism $$\pi_1(X_{\text{red}},x)\to \pi_1(X,x)$$ between the fundamental group schemes of $X_{\text{red}}$ and $X$  induced by $i_{\text{red}}$ is a closed immersion. 
\end{theorem}

\section{Construction}
Throughout this  section  $S$ will be any connected Dedekind scheme  and $f:X\to S$ a faithfully flat morphism of finite type of schemes endowed with a fixed $S$-valued point $x:S\to X$. Moreover a triple $(Y,G,y)$ over $X$ will always stand for a \textit{fpqc}-torsor $Y\to X$, under the (right) action of a  flat $S$-group scheme $G$ endowed with a $S$-valued point $y\in Y_x(S)$ and a morphism between two such triples $(Y,G,y)\to(Y',G',y')$ will be  the datum of a pair of morphisms $\alpha:G\to G'$, $\beta:Y\to Y'$ such that $\beta(y)=y'$ and such that the following diagram commutes

$$\xymatrix{Y\times G \ar[r]^{\beta\times \alpha} \ar[d]_{G\text{-}action} & Y'\times G'\ar[d]^{G'\text{-}action}\\ Y\ar[r]_{\beta} & Y'}$$  The category whose objects are isomorphism classes of triples $(Y,G,y)$ with the additional assumption that $G$ is finite and flat is denoted by $\mathcal{P}(X)$. The aim of this paper is to construct the fundamental group scheme of $X$ in $x$ following the usual definition (cf. \cite{Nor2}, Chapter II, Definition 1) extended here to a Dedekind scheme:

\begin{definition}The scheme  $X$ has a fundamental group scheme $\pi_1(X,x)$ if there exists a triple $(\widetilde{Y},\pi_1(X,x),\widetilde{y})$ such that for  any $(Y,G,y)\in Ob(\mathcal{P}(X))$ there is a unique morphism $(\widetilde{Y},\pi_1(X,x),\widetilde{y})\to (Y,G,y)$. The $\pi_1(X,x)$-torsor $\widetilde{Y}\to X$ is called the universal torsor.
\end{definition}

%

\begin{definition}\label{defDomi} We say that a triple $(Y,G,y)$ over $X$ is dominated if for  every morphism $(Y',G',y')\to (Y,G,y)$ the group morphism $G'\to G$ is schematically dominant, i.e. the corresponding morphism on their Hopf algebras is injective.
\end{definition}

\begin{definition}\label{defQuoz} We say that a triple $(Y,G,y)$ over $X$ is quotient if for  every morphism $(Y',G',y')\to (Y,G,y)$ the group morphism $G'\to G$ is faithfully flat.
\end{definition}

\begin{definition}\label{defPreced} We say that a triple $(Y,G,y)$ over $X$ is preceded by a triple $(Y',G',y')$ if there exists a morphism $(Y',G',y')\to (Y,G,y)$.
\end{definition}

While in general we can only say that a quotient triple is also a dominated one, it is  clear that when $S$ is the spectrum of a field then  quotient triples and dominated triples coincide (cf. for example \cite{Antei} \S 1.1 and the references therein). These were called \textit{reduced} by Nori in \cite{Nor2} and often in recent literature they are called \textit{Nori-reduced}.

%

\begin{lemma}\label{lemmaclosed}  Let $G$ and $H$ be two finite and flat $S$-group schemes both of order $n$. Assume that $i:H\hookrightarrow G$ is a closed immersion then $i$ is an isomorphism.
\end{lemma}
\proof Easy, since $S$ is a Dedekind scheme.
\endproof

\begin{proposition}\label{propclosed} Every triple $(Y,G,y)$ is preceded by a dominated triple. 
\end{proposition}
\begin{proof} If $(Y,G,y)$ is already dominated there is nothing to prove. Otherwise there exists a triple $(Y_1,G_1,y_1)$ and a closed immersion $(Y_1,G_1,y_1)\hookrightarrow (Y,G,y)$ (i.e. $G_1\hookrightarrow G$ and consequently $Y_1 \hookrightarrow Y$ are closed immersions); indeed since $(Y,G,y)$ is not dominated there exist at least a triple $(Y'_1,G'_1,y'_1)$ and a morphism $(Y'_1,G'_1,y'_1)\to (Y,G,y)$ which is not schematically dominant, i.e. the canonical morphism $p_1:G'_1\to G$ is not schematically dominant. Thus $p_1$ factors through a closed immersion $G'_1\to G_1\hookrightarrow G$ (cf. for example \cite{Antei}, Lemma 2.2).
We simply say that $(Y_1,G_1,y_1)$ is contained in $(Y,G,y)$, where $Y_1:=Y_1'\times^{G'_1}G_1$. Between all the triples $(Y_i,G_i,y_i)$ contained in $(Y,G,y)$ we can choose, since $|G|$ is finite, one triple $(Y',G',y')$ such that $n:=|G'|$ is the smallest possible. We claim that this triple is dominated. If it was not then it would contain a triple $(P,H,p)$ with $|H|\leq n$. But $(P,H,p)$  is also contained in $(Y,G,y)$ then by the minimality of $n$ we have $|H|=n$, hence the canonical morphism $H\to G'$ being a closed immersion is an isomorphism according to lemma \ref{lemmaclosed}. Thus $(Y',G',y')$ is a dominated triple preceding $(Y,G,y)$.
\end{proof}

\subsection{Schemes over a field}

If $S=Spec(k)$ with $k$ any field, we have already observed that proposition \ref{propclosed} implies that every triple over $X$ is preceded   by a quotient triple. Then it is now easy to prove that the category of quotient triples is filtered. More precisely: let  $\mathcal{P}_q(X)$ be  the full subcategory of $\mathcal{P}(X)$ whose objects are isomorphism classes of  quotient triples.  Then the following theorem holds:

\begin{theorem}\label{teoGenerFundGrupK} Let $k$ be a field, $S=Spec(k)$, $X$ a scheme and $f:X\to S$ a faithfully flat morphism of finite type provided with a section $x\in X(S)$. Then the category
$\mathcal{P}_q(X)$ is filtered.
 \end{theorem} 
\proof 
Given three quotient triples $(Y,G,y)$ and $(Y_i,G_i,y_i)$, i=1,2,   with (faithfully flat) morphisms $\gamma_i:(Y_i,G_i,y_i)\to (Y,G,y)$ we need to prove the existence of a quotient triple $(T',H',t')$ with maps $(T',H',t')\to (Y_i,G_i,y_i)$ making the following diagram 
$$\xymatrix{(T',H',t')\ar[r]\ar[d] & (Y_1,G_1,y_1)\ar [d]\\(Y_1,G_1,y_1)\ar[r]& (Y,G,y)}$$ commute. First we prove that $(T,H,t):=(Y_1\times_Y Y_2,G_1\times_G G_2,y_1\times_y y_2)\in Ob(\mathcal{P}(X))$. Indeed pulling back by $T\to X$, which is faithfully flat, we obtain the diagram
$$\xymatrix{ & T\times_X T\ar[dr]\ar[dl]\ar[rrrr]^{\alpha} & & & & T\ar[dr]\ar[dl] & \\
T\times_X Y_1 \ar[dr]\ar@/^/[rrrr] & & T\times_X Y_2 \ar[dl]\ar@/_/[rrrr]& & Y_1 \ar[dr] & & Y_2 \ar[dl] \\ 
 & T\times_X  Y \ar[d]\ar[rrrr] & & & & Y \ar[d] & \\
 & T \ar[rrrr] & & & & X & } $$

\noindent but $T$ is a fpqc covering of $X$ trivializing each of $Y$, $Y_1$ and $Y_2$ thus from $T\times G_i \simeq T\times_X Y_i$ ($i=1,2$) and $T\times G \simeq T\times_X Y$ we obtain an isomorphism $ \psi: T\times H\simeq T\times_X T$. Then the action $\alpha\circ \psi: T\times H\to T$ of $H$ on $T$ gives $T$ the desired structure of $H$-torsor.   If $(T,H,t)$ is quotient we are done. Otherwise we use proposition \ref{propclosed} in order to obtain a  quotient triple $(T',H',t')$ preceding $(T,H,t)$; this concludes the proof.
\endproof

Then we can define a
pro-object $\underleftarrow{lim}_{
\mathcal{P}_q(X) } (Y_i,G_i,y_i)$. We still denote $\underleftarrow{lim}_{
\mathcal{P}_q(X) }G_i$ by $\pi_1(X,x)$ which is a $S$-group scheme (cf. \cite{EGAIV-3}, Proposition 8.2.3)
and $\underleftarrow{lim}_{
\mathcal{P}_q(X) }Y_i$ by $\widetilde{Y}$ which is a scheme pointed in $\widetilde{y}:=\underleftarrow{lim}_{
\mathcal{P}_q(X) }y_i$. Now for any $(Y,G,y)\in {
\mathcal{P}_q(X) }$,  according to proposition \ref{propclosed} and theorem \ref{teoGenerFundGrupK}, there is a (necessarily unique)  morphism  $$(\widetilde{Y},\pi_1(X,x),\widetilde{y})\to(Y,G,y),$$  thus $X$ has a fundamental group scheme:

\begin{definition}\label{defGenerFundGrup}  Let $k$ be a field. We call the $k$-group scheme $\pi_1(X,x)$
  the fundamental group scheme of $X$ in $x$ and  $\widetilde{Y}$ the $\pi_1(X,x)$-universal torsor
over $X$.\end{definition}

\subsection{Schemes over a Dedekind scheme}

While for schemes over a field it has been quite easy for schemes over a Dedekind scheme it will be a little more complicated since in general  a triple is not preceded  by a  quotient triple. So we work in the category $\mathcal{P}_d(X)$, the full subcategory of $\mathcal{P}(X)$ whose objects are isomorphism classes of dominated triples over $X$ and we will use the results of previous section. In this section $S$ will be any connected Dedekind scheme. The proof of the following theorem, even if slightly different, is strongly based on the proofs of \cite{Nor2}, II, Lemma 1 and \cite{GAS}, Proposition 2.1, but the details are recalled, for the sake of completeness, and simplified where possible. The only completely new point is step 3 of the proof.

\begin{theorem}\label{teoGenerFundGrupS} Let $S$ be a connected Dedekind scheme, $X$ a connected scheme and $f:X\to S$ a faithfully flat morphism of finite type provided with a section $x\in X(S)$. Then the category
$\mathcal{P}_d(X)$ is filtered. \end{theorem} 
\proof 
In what follows $\eta$ will denote the generic point of $S$. Given three dominated triples $(Y,G,y)$ and $(Y_i,G_i,y_i)$, i=1,2,   with morphisms $\gamma_i:(Y_i,G_i,y_i)\to (Y,G,y)$ we need to prove the existence of a dominated triple $(T',H',t')$ with maps $(T',H',t')\to (Y_i,G_i,y_i)$ making the following diagram 
$$\xymatrix{(T',H',t')\ar[r]\ar[d] & (Y_2,G_2,y_2)\ar [d]\\(Y_1,G_1,y_1)\ar[r]& (Y,G,y)}$$ commute. First we prove that $(T,H,t):=(Y_1\times_Y Y_2,G_1\times_G G_2,y_1\times_y y_2)\in Ob(\mathcal{P}(X))$. We can assume $H$ to be a finite and flat $S$-group scheme (if it is not we can replace $(T,H,t)$  by the schematic closure of its generic fiber $(T_{\eta},H_{\eta},t_{\eta})$ in $(T,H,t)$, according to \cite{GAS}, Lemma 2.2). We divide the reminder of the proof in three steps:
\begin{enumerate}
	\item   $T$ is a $H$-torsor over $T/H$: the morphism $Y\to X$ is separated (since affine) then $T$ is a closed subscheme of $E:=Y_1\times_X Y_2$, which is clearly a $G_1\times_S G_2$-torsor. Consider the diagram

$$\xymatrix{ & E\times_X T\ar[dr]\ar[dl]\ar[rrrr] & & & & T\ar[dr]\ar[dl] & \\
E\times_X Y_1 \ar[dr]\ar@/^/[rrrr] & & E\times_X Y_2 \ar[dl]\ar@/_/[rrrr]& & Y_1 \ar[dr] & & Y_2 \ar[dl] \\ 
 & E\times_X  Y \ar[d]\ar[rrrr] & & & & Y \ar[d] & \\
 & E \ar[rrrr] & & & & X & } $$
but  $E$ trivializes all the three torsors $Y\to X$ and $Y_i\to X$ since $E\to X$ factors through $Y$ but also through $Y_i$.	Thus from the isomorphisms $E\times_X Y_i\simeq E\times_S G_i$ and $E\times_X Y\simeq E\times_S G$ we obtain the isomorphism $E\times_X T\simeq E\times_S (G_1\times_G G_2)$. Pulling back by the closed immersion $T\hookrightarrow E$ we finally obtain the isomorphism $T\times_X T\simeq T\times_S (G_1\times_G G_2)=T\times_S H$. Thus $H$ acts freely on $T$. Since we do not know whether $T$ is faithfully flat over $X$ we can only deduce that $T$ is a $H$-torsor over $T/H$ and not the whole $X$.  
	\item The canonical morphism $T/H\to X$ is a closed immersion: since it is a finite morphism it is sufficient to prove that the diagonal morphism $\Delta:=\Delta_{(T/H)/X}: T/H \to T/H\times_{X}T/H$ is an isomorphism.  Consider  the isomorphism $T\times_{T/H}T\simeq T\times_{X}T$ given by, for every $S$-scheme  $U$,  the composition of $$T\times_{T/H}T\to T\times_S H, \qquad (t,t')\mapsto (t,h)$$ 
for  $t,t'\in T(U)$ and $h\in H(U)$ the unique element such that $t\cdot h=t'$ and $$T\times_S H\to T\times_{X}T,\qquad (t,h)\mapsto (t,t\cdot h)=(t,t').$$ Then it can be rewritten as 
	$$id_T\times\Delta\times id_T: T\times_{T/H}T/H \times_{T/H}T \to T\times_{T/H}(T/H\times_X T/H)\times_{T/H}T.$$ 
Thus $\Delta_{(T/H)/X}$ is an isomorphism as required.
	\item The above closed immersion $T/H\hookrightarrow X$ is an isomorphism: observe that $T/H$ is flat over $S$ since $T$ is flat over $S$ and over $T/H$. So let  $X_{\eta}$ be the generic fiber of $X$, then the triple $(T_{\eta},H_{\eta},t_{\eta})$, generic fiber of $(T,H,t)$, is a torsor over $X_{\eta}$ as we have already proved: indeed it is the fibered product of the triples $({Y_i}_{\eta}, {G_i}_{\eta}, {y_i}_{\eta})$ over $({Y}_{\eta}, {G}_{\eta}, {y}_{\eta})$ and the morphisms between such triples are all faithfully flat. Moreover $X_{\eta}\simeq T_{\eta}/H_{\eta}\simeq (T/H)_{\eta}$ where the last isomorphism comes from \cite{Dem}, Proposition 3.4.5. As $T/H$ is the unique subscheme of $X$, flat over $S$ whose generic fiber is isomorphic to $X_{\eta}$ it follows that $T/H\simeq X$ (cf. \cite{EGAIV-2}, Proposition 2.8.5) thus $(T,H,t)\in  Ob(\mathcal{P}(X))$. If $(T,H,t)$ is not a dominated triple simply use proposition \ref{propclosed} in order to conclude.
\end{enumerate}

\endproof

As in previous section we can construct $\pi_1(X,x)$, $\widetilde{Y}$ and $\widetilde{y}$ in this more general setting.

\begin{definition}\label{defGenerFundGrupDed}  Let $S$ be a connected Dedekind scheme, $X$ a connected scheme and $f:X\to S$ a faithfully flat morphism of finite type. We call the $S$-group scheme $\pi_1(X,x)$  the fundamental group scheme of $X$ in $x$ and  $\widetilde{Y}$ the $\pi_1(X,x)$-universal torsor
over $X$.\end{definition}

\section{Comparison of fundamental group schemes after thickening}
\label{sez:Comparison}

\subsection{The cotangent complex of a morphism of schemes}

Let  $f:X\to Y$ be a morphism of schemes and let  $L^{\bullet}_{Y/X}$ its cotangent complex (cf.  \cite{ILLU1}, II, \S 1.2 and \cite{LICSCH}): it is a complex of $\mathcal{O}_X$-modules of perfect amplitude $\subset [-1,0]$ (cf. \cite{ILLU3}, \S 1). Now consider the following situation: let $S$ be any base scheme, $G$ a flat $S$-group scheme locally of finite presentation, let $X$ and $Y$ be $G$-schemes and $f:Y\to X$  a $G$-equivariant morphism of schemes (e.g. $f$ is a $G$-torsor), then we can define a complex of $G$-equivariant $\mathcal{O}_X$-modules $(L^{G}_{Y/X})^{\bullet}$ whose underlying complex of $\mathcal{O}_X$-modules is isomorphic to $L_{Y/X}^{\bullet}$ (cf. \cite{ILLU3}, \S 2): it is called the equivariant cotangent complex of $Y$ over $X$. Finally we denote by $l_{Y/X}$ the co-Lie complex of $Y$ over $X$ (cf. \cite{ILLU3}, Definition 2.5).

\subsection{Comparison}
From now on $S$ will denote a connected Dedekind scheme. By a thickening of order one we mean a closed immersion of schemes $i:X_0\to X_1$ whose sheaf of ideal $\mathcal{I}:=ker(i^{\sharp}:\mathcal{O}_{X_1}\twoheadrightarrow i_{\ast} \mathcal{O}_{X_0})$ is of square zero, i.e. $\mathcal{I}^2=0$; then in particular  the underlying topological spaces are the same. Let $G$ be a finite and flat $S$-group scheme, by a deformation (or extension) of a $G$-torsor  $f_0:Y_0\to X_0$ we mean  a $G$-torsor $f_1:Y_1\to X_1$ whose pull back over $X_0$ is isomorphic to $f_0:Y_0\to X_0$. We recall the following result concerning  deformation of torsors after thickening stated in our settings:

\begin{theorem}\label{theoThick} Let $i:X_0\to X_1$ and $\mathcal{I}$ be as before. Let $G$ be a finite and flat $S$-group scheme and $f_0:Y_0\to X_0$ a $G$-torsor, then
\begin{enumerate}
	\item there exists an obstruction $\omega(f_0,i)\in \mathbb{H}^2(X_{0},l^{\vee}_{Y_0/X_0}\otimes^L \mathcal{I})$ whose vanishing is necessary and sufficient for the existence of a deformation of $f_0:Y_0\to X_0$.
	\item When $\omega(f_0,i)=0$ then the set of isomorphism classes of $G$-torsors over $X_1$  deforming $Y_0\to X_0$ is an affine space under the action of the group $\mathbb{H}^1(X_{0},l^{\vee}_{Y_0/X_0}\otimes^L \mathcal{I}).$
\end{enumerate}
\end{theorem}
\proof Cf. \cite{ILLU2}, Th\'eor\`eme 2.4.4 and Remarque 2.4.4.1 a).
\endproof

\begin{lemma}\label{lemThick} Let $T$ be any affine scheme, $\mathcal{F}^{\bullet}$ a complex of $\mathcal{O}_T$-sheaves of modules concentrated in $[0,1]$, then $\mathbb{H}^n(T,\mathcal{F}^{\bullet})=0$ for every $n\geq 2$. 
\end{lemma}
\proof Recall that the hypercohomology group $\mathbb{H}^n(T,\mathcal{F}^{\bullet})$ is defined as $\mathbb{R}^n\Gamma(\mathcal{F}^{\bullet})=H^n(Tot^{\prod}(\Gamma(\mathcal{I}^{\bullet\bullet})))$ where $\Gamma$ is the global section functor and $\mathcal{I}^{\bullet\bullet}$ is a Cartan-Eilenberg resolution of $\mathcal{F}^{\bullet}$. Since $\mathcal{F}^{\bullet}$ is concentrated in $[0,1]$ and $T$ is affine then the result follows from simple computations on the morphisms 
$$d^n:Tot^{\prod}(\Gamma(\mathcal{I}^{\bullet\bullet}))^n\to Tot^{\prod}(\Gamma(\mathcal{I}^{\bullet\bullet}))^{n+1}.$$ 
\endproof

We now state a useful consequence of previous facts:

\begin{corollary}\label{corolThick} Let $i:X_0\to X_1$ and $\mathcal{I}$ be as in theorem \ref{theoThick} with the additional assumption that $X_0$ and $X_1$ are affine. Let $G$ be a finite and flat $S$-group scheme and $f_0:Y_0\to X_0$ a $G$-torsor, then $f_0:Y_0\to X_0$ admits a deformation over $X_1$. 
\end{corollary}
\begin{proof}According to lemma \ref{lemThick} the hypercohomology group $\mathbb{H}^2(X_{0},l^{\vee}_{Y_0/X_0}\otimes^L \mathcal{I})$ is trivial, then we conclude using theorem \ref{theoThick}, point 1.
\end{proof}

Now let $S$ be an affine and connected Dedekind scheme, $X:=Spec(A)$ a connected, affine and noetherian scheme, faithfully flat and of finite type over $S$ and consider the canonical closed immersion $$i_{\text{red}}:X_{\text{red}}\hookrightarrow X$$ where $X_{\text{red}}:=Spec(A/I)$ denotes the reduced part of $X$. Since $X$ is noetherian then the nilradical $I$ of $A$ is nilpotent, i.e. there exists an integer $n>0$ such that $I^n=0$ and $I^{n-1}\neq 0$ (\cite{Isaacs}, (14.38) Theorem (Levitsky)). For all $j=2, .., n$ set 
$I_j:=\{i\in I | i\cdot i_1\cdot .. \cdot i_{j-1}=0 ~ \forall i_1, .. , i_{j-1}\in I\}$ then $(I_j)^j=0$, $(I_j)^{j-1}\neq 0$ and $I_n=I$ thus  $i_{\text{red}}:X_{\text{red}}\hookrightarrow X$ can be factored into  a finite number of thickenings of order one:

$$Spec(A/I_n)\hookrightarrow Spec(A/I_{n-1})\hookrightarrow .. \hookrightarrow Spec(A/I_{2}) \hookrightarrow Spec(A)$$

\begin{theorem}\label{theoRed} Let $X$ be a connected, affine and noetherian scheme, faithfully flat and of finite type over $S$ and let $i_{\text{red}}:X_{\text{red}}\hookrightarrow X$ be the canonical closed immersion. Let $x\in X_{\text{red}}(S)$, then the natural morphism $$\pi_1(X_{\text{red}},x)\to \pi_1(X,x)$$ between the fundamental group schemes of $X_{\text{red}}$ and $X$  induced by $i_{\text{red}}$ is a closed immersion.
\end{theorem}
\proof It  follows by previous discussion and  corollary \ref{corolThick} iterated $n-1$ times. 
\endproof

As observed in the introduction this closed immersion is an isomorphism when $char(k(s))=0$ for every $s\in S$.

\section{Conclusion}
In this paper we have constructed the fundamental group scheme of a scheme $X$ over a connected Dedekind scheme $S$ even if $X$ is not integral. Then we have compared, when $X$ is affine, the fundamental group scheme of a scheme $X$ with that of $X_{\text{red}}$. We do not know how easy can be to study the morphism $\pi_1(X_{\text{red}},x)\to \pi_1(X,x)$ when $X$ is not affine. It would be interesting to investigate, for instance, the behavior of the above morphism when $X$ is a projective irreducible, non reduced curve over a field $k$ of positive characteristic;   as shown, the cotangent complex offers a useful tool for further analysis.

\indent \textbf{Acknowledgements:} I would like to thank Michel Emsalem for his interest in this paper and Lorenzo Ramero whose  very interesting remarks helped me improving the paper.

\begin{flushright}Marco Antei\\ 
Dept. of Mathematical Science,\\ 335 Gwahangno (373-1 Guseong-dong),\\ Yuseong-gu, Daejeon 305-701,\\
Republic of Korea
\end{flushright}

\begin{flushright} E-mail:\\  \texttt{marco.antei@gmail.com}
\end{flushright}

\end{document}